\documentclass[12pt]{amsart}
\usepackage{multirow,bigdelim}
\usepackage{amsfonts}
\usepackage{dsfont}
\usepackage{amsmath}
\usepackage{mathrsfs}
\usepackage{todonotes}
\usepackage{hyperref}
\usepackage{multimedia}
\usepackage{graphicx}
\usepackage{indentfirst}
\usepackage{bm}
\usepackage{amsthm}
\usepackage{mathrsfs}
\usepackage{latexsym}
\usepackage{amsmath}
\usepackage{amssymb}
\usepackage{microtype}
\usepackage{relsize}
\usepackage{hyperref}

\newcommand{\A}{\mathcal{A}}

\newcommand{\N}{\mathcal{N}}

\def\A{\mathcal A}
\def\B{\mathcal B}

\def\B{\mathcal B}

\def\H{\mathcal H}

\def\M{\mathcal M}
\def\N{\mathcal N}

\def\NNN{\mathbb N}

\def\amslatex{$\mathcal{A}\kern-.1667em\lower.5ex\hbox{$M$}\kern-.125em\mathcal{S}$-\LaTeX}

\DeclareSymbolFont{SY}{U}{psy}{m}{n}
\DeclareMathSymbol{\emptyset}{\mathord}{SY}{'306}

\theoremstyle{plain}

\newtheorem{thm}{Theorem}[section]
\newtheorem{cor}[thm]{Corollary}
\newtheorem{lem}[thm]{Lemma}
\newtheorem{prop}[thm]{Proposition}
\theoremstyle{definition}
\newtheorem{defn}[thm]{Definition}
\newtheorem{rem}[thm]{Remark}

\numberwithin{equation}{section}

\begin{document}
\title{On finite sums of projections and Dixmier's averaging theorem for type ${\rm II}_1$ factors}
\author{ Xinyan Cao, Junsheng Fang and Zhaolin Yao}
\curraddr[X. Cao and J. Fang]{School of Mathematical Sciences, Hebei Normal University,
Shijiazhuang, Hebei 050016, China}
\curraddr[Z. Yao]{Postdoctoral Research Station of Mathematics, Hebei Normal University, Shijiazhuang, Hebei 050016, China}
\email[X. Cao]{caoxinyan123@qq.com}
\email[J.Fang]{jfang@hebtu.edu.cn}
\email[Z.Yao]{zyao@hebtu.edu.cn}
\email[Z. Yao]{}

\thanks{The authors are supported  by National Natural
Science Foundation of China (Grant No.12071109) and a start up funding from Hebei Normal Univeristy.}

\subjclass[2010]{46L10, 46L36} \keywords{finite sums of projections, Dixmier's averaging theorem, type ${\rm II}$ factors}
\begin{abstract}
Let $\M$ be a type ${\rm II_1}$ factor and let $\tau$ be the faithful normal tracial state on $\M$. In this paper, we prove that given an $X \in \M$, $X=X^*$, then there is a decomposition
of the identity into $N \in \NNN$ mutually orthogonal nonzero projections $E_j\in\M$, $I=\sum_{j=1}^NE_j$, such that $E_jXE_j=\tau(X) E_j$ for all $j=1,...,N$. Equivalently,  there is a
unitary operator $U \in \M$ with $U^N=I$ and
$\frac{1}{N}\sum_{j=0}^{N-1}{U^*}^jXU^j=\tau(X)I.$ As the first application, we prove that a positive operator $A\in \M$ can be written as a finite sum of projections in $\M$ if and only if
$\tau(A)\geq \tau(R_A)$, where $R_A$ is the range projection of $A$. This result answers affirmatively Question 6.7 of \cite{HKNZ}. As the second application, we show that if $X\in \M$,
$X=X^*$ and $\tau(X)=0$, then there exists a nilpotent element $Z \in \M$ such that $X$ is the real part of $Z$. This result answers affirmatively Question 1.1 of \cite{DFS}. As the third
application, we show that for $X_1,...,X_n\in \M$, there exist unitary operators $U_1,...,U_k\in\M$ such that
$
\frac{1}{k}\sum_{i=1}^kU_i^*X_jU_i=\tau(X_j)I,\quad \forall 1\leq j\leq n.
$
This result is a stronger version of Dixmier's averaging theorem for type ${\rm II}_1$ factors.

\end{abstract}

\maketitle
\section{Introduction}

Which positive bounded linear operators on a separable Hilbert space can be written as strong sums of projections? If the underlying space is finite dimensional, then a characterization of such
operators was obtained by Fillmore~\cite{Fil}: a finite-dimensional operator is the sum of projections if and only if it is positive, it has an integer trace and the trace is greater than or
equal to its rank.

The same question can be asked relative to a von Neumann algebra $\M$. This question was first asked by Kaftal, Ng and Zhang in~\cite{KNZ}. In~\cite{KNZ}, they developed
beautiful techniques and addressed the question completely for positive operators in type ${\rm I}$ and type ${\rm III}$ factors. They proved the following results. Let $\M$ be a type ${\rm
I}$ factor. Given $A \in \M^{+}$. Denote by $A_+=(A-I)\chi_A(1,\|A\|]$ the excess part of $A$ and by $A_-=(I-A)\chi_A(0,1)$ the defect part of $A$. Then $A$ is a strong sum of projections if and only if either ${\rm Tr}(A_+)=\infty$ or ${\rm Tr}(A_-)\leq {\rm Tr}(A_+)<\infty$ and ${\rm Tr}(A_+)-{\rm Tr}(A_-)\in
\NNN\cup\{0\}$. Let $\M$ be a type ${\rm III}$ factor. Then $A\in\M$ is a strong sum of projections if and only if either $\|A\|>1$ or $A$ is a projection. For type ${\rm II}$ factors, they proved that if $A$ is
diagonalizable, then $A$ is a strong sum of projections if and only if  $\tau(A_+)\geq \tau(A_-)$. The condition  $\tau(A_+)\geq \tau(A_-)$ is necessary even when $A$ is not diagonalizable.
Recall that a positive operator $A$ is said to be diagonalizable if $A=\sum \gamma_j E_j$ for some $\gamma_j>0$ and mutually orthogonal projections $\{E_j\}$ in $\M$.
 In \cite{CFY}, we proved that if $\M$ is a type ${\rm II}$ factor, $A\in \M^+$ and $\tau(A_+)\geq \tau(A_-)$, then $A$ can be written as the sum of a finite or infinite collection of
 projections.

 A natural question is the following: which positive operators in a factor von Neumann algebra can be written as finite sums of projections? This question is much more interesting and difficult. Note that the question is even open
 for positive operators in $\B(\H)$. One of the main goals of this paper is to study this question in type ${\rm II}_1$ factors.

As reported in a survey article [\cite{Wu}, Theorem 4.12] by Wu quoting unpublished joint work of Choi and Wu
in 1988, they proved that positive operators with essential norm greater than 1 are finite sums of projections.
This result is presented together with other related results in their recent paper [\cite{C-W}, Theorem 2.2]. The special
case of $\alpha I$ with $\alpha>1$ follows also from a delicate analysis which characterizes the numbers $\alpha$ for which $\alpha I$
is a sum of at least $n$ projections (see in ~\cite{KRS} and several other papers.)

\cite{KNZ3} presented a characterization of strict sums of projections in the multiplier algebras of $\sigma$-unital purely
infinite simple $C^*$-algebras and~\cite{KNZ1} provided sufficient conditions for operators to be finite sums of projections
in the multiplier algebras of $\sigma$-unital purely infinite simple $C^*$-algebras.

In \cite{HKNZ}, H. Halpern, V. Kaftal, P. Ng and S. Zhang provided some sufficient conditions and also a necessary condition for a positive operator can be written as a finite sum of
projections. For a type ${\rm II}_1$ factor, they showed a sufficient condition for a positive diagonalizable operator $A$ to be a finite sum of projections is $\tau(A_+)>\tau(A_-)$ or
equivalently $\tau(A)>\tau(R_A)$.

In this paper, we prove the following result.
\begin{thm}[Main Theorem]

Let $(\M,\tau)$ be a type ${\rm II}_1$ factor and let $X=X^*\in\M$. Then
\begin{enumerate}
\item there is a decomposition of the identity into $N\in \NNN$ mutually orthogonal nonzero projections $E_j$, $I=\sum_{j=1}^NE_j$, for which $E_jXE_j=\tau(X) E_j$ for all $j=1,...,N$;
\item there is a finite dimensional abelian von Neumann algebra $\B \subseteq \M$ such that
      \[E_{\B^\prime \cap \M}(X)=\tau(X)I, \]
      where $E_{\B^\prime \cap \M}$ is the conditional expectation from $\M$ onto $\B^\prime \cap \M$;
\item there is a unitary operator $U \in \M$ with
\[\frac{1}{N}\sum_{j=0}^{N-1}{U^*}^jXU^j=\tau(X)I;\]
\item there is a unitary operator $W \in \M$ with $W^N=I$ and
\[\frac{1}{N}\sum_{j=0}^{N-1}{W^*}^jXW^j=\tau(X)I.\]
      \end{enumerate}
\end{thm}

This result is inspired by \cite{C-W}. To prove this theorem, we develop some new techniques in this paper. On the other hand, several results, e.g., Proposition 2.3, Corollary 2.5, Lemma
2.8 and Lemma 6.1, in~\cite{HKNZ} play important roles in the proof.

We mention the following results as applications of the main theorem.

\begin{thm}
Let $(\M,\tau)$ be a type ${\rm II}_1$ factor and $A\in \M^+$. Then $A$ is a finite sum of projections if and only if $\tau(A)\geq \tau(R_A)$.
\end{thm}

 The above theorem  answers affirmatively Question 6.7 of \cite{HKNZ}. The following theorem answers affirmatively Question 1.1 in \cite{DFS}.

\begin{thm}
Let $(\M,\tau)$ be a type ${\rm II}_1$ factor and $X\in \M$, $X=X^*$ and $\tau(X)=0$. Then there exists a nilpotent element $Z \in \M$ such that $X=Re Z$, i.e., $X$ is the real part of a nilpotent element $Z$.
\end{thm}

The following theorem is a stronger version of Dixmier's averaging theorem for type ${\rm II}_1$ factors.
\begin{thm}
Let $(\M,\tau)$ be a type ${\rm II}_1$ factor and let $X_1,...,X_n\in \M$. Then there exists unitary operators $U_1,...,U_k\in\M$ such that
\[
\frac{1}{k}\sum_{i=1}^kU_i^{*}X_jU_i=\tau(X_j)I,\quad \forall 1\leq j\leq n.
\]
\end{thm}

The organization of this paper is as follows. Section 2 below is devoted to some definitions and lemmas which will be very useful. In section 3, we prove the main theorem stated as above. As
applications, we prove Theorem 1.2 and Theorem 1.3. In section 4, we give some properties of majorization and provide the third application of the main
theorem.

We refer to~\cite{C-W,G-P,KNZ1,KNZ2,KNZ3,KNZ4,KRS2,Mar,Nak} for interesting results on the related topics. We refer to~\cite{K-R,S-S,S-Z,Tak} for the basic knowledge of von Neumann
algebras.

\emph{Notation}. For every $x\in \mathbb{R}$, we will denote by $\lceil x \rceil$ the smallest integer $n\geq x$.

Throughout this paper, $(\M,\tau)$ is a type ${\rm II}_1$ factor with the faithful normal tracial state $\tau$. $E,F,E^{\prime},F^{\prime},p,q$ will denote projections in $\M$. $\M^+$ is the set of positive
operators in $\M$. $\NNN$ denotes the set of non-negative integers and $\NNN^*$ is the set of strictly positive integers.

For a selfadjoint element $A \in \M$, $\chi_A$ will denote the (projection-valued) spectral measure of $A$ and $R_A$ the range projection of $A$.

\section{Definitions and Preliminary lemmas}

\begin{defn}[Definition 2.2 in \cite{HKNZ}]
An operator algebra $\M$ has constants $N_0$ and $V_0$ if every selfadjoint operator $a\in \M$ can be decomposed into a real linear combination $a=\sum_{j=1}^{N_0}\alpha_j p_j$ of $N_0$
projections $p_j\in \M$ such that
\[
\sum_{j=1}^{N_0}|\alpha_j| \leq V_0 \|a\|.
\]
\end{defn}

The existence and the estimates of the constants $N_0$ and $V_0$ for $\rm II_1$ factors are given in Theorem 1-3 in \cite{G-P}.
If $\M$ is of type $\rm II_1$ factor, then $N_0=12$ and $V_0=14$.

\begin{defn}
A linear combination is called positive combinations of projections if a positive operator is finite linear combination of projections with positive coefficients.
\end{defn}

\begin{prop}[Proposition 2.3 in \cite{HKNZ}]\label{prop1: proposition 2.3}
Let $\M$ be a type $\rm II_1$ factor. Assume that $a \in \M^+$  is invertible. Then $a$ is a positive combination of $N_0+\lceil V_0(\|a\|\|a^{-1}\|-1)\rceil+1$ projections in
$\M$.
\end{prop}

\begin{defn}
A positive operator $a$ such that $a \geq \nu R_a$ for some $\nu>0$ is invertible in the reduced algebra $R_a\M R_a$. We call such an element \emph{locally invertible}.
\end{defn}

\begin{lem}[Lemma 2.4 in \cite{HKNZ}]\label{lem1: lemma 2.4}
Let $a\in \M^+$. Then $a$ is a positive combination of projections in $\M$ if and only if $a$ is a positive combination of projections in $R_a \M R_a$.

\end{lem}

\begin{cor}[Corollary 2.5 in \cite{HKNZ}]\label{cor1: Corollary 2.5}
Let $a \in \M$ and assume that it is locally invertible, i.e., $a \geq \nu R_a$ for some $\nu > 0$. Then $a$ is a positive combination of
$N_0+\lceil V_0(\frac{\|a\|}{\nu}-1)\rceil+1$ projections in $\M$.
\end{cor}

\begin{lem}[Lemma 2.8 in \cite{HKNZ}]\label{lem1: lemma 2.8}
Let $e$ and $f$ be orthogonal projections in $\M$ with $e \prec f$. Let $b=be=eb$, $d=df=fd$ be positive elements in $\M$ such that $d \geq \nu f$ for some $\nu >0$. Then $a:=b+d$ is a
positive combination of projections in $\M$.

\end{lem}

\begin{lem}[Lemma 6.1 in \cite{HKNZ}]\label{lem1: lemma 6.1}
Let $e$ and $f$ be mutually orthogonal projections in $\M$, $0 \leq \lambda \leq 1$, and $\mu >0$, and $a:=(1+\mu)e+(1-\lambda)f$. If $\mu \tau(e)> \lambda \tau(f)$, then $a$ is a finite
sum
of projections in $\M$.
\end{lem}

\begin{lem}[Lemma 3.2 in \cite{CFY}]\label{lem1: lemma 3.2}
Let $(\M,\tau)$ be a type ${\rm II}_1$ factor and let $a\in \M^+$ such that $\tau(a)=1$. Then $a$ is a strong sum of projections.
\end{lem}

\section{Main theorem}

The following lemma is motivated by Proposition 1.5 of \cite{C-W} and its proof is similar to the proof of Proposition 1.5 of \cite{C-W}. We  provide a complete proof  for the reader's convenience.

\begin{lem}\label{lem2.1:TFAE}
Let $X\in \M$. Then the following conditions are equivalent.
\begin{enumerate}
\item There is a decomposition of the identity into $N \in \NNN^*$ mutually orthogonal nonzero projections $E_j$, $I=\sum_{j=1}^NE_j$, for which $E_jXE_j=\tau(X) E_j$ for all $j=1,...,N$;
\item There is a unitary operator $U \in \M$ with
\[\frac{1}{N}\sum_{j=0}^{N-1}{U^*}^jXU^j=\tau(X)I.\]
\end{enumerate}
\end{lem}
\begin{proof}
(1)$\Rightarrow$(2). Set $A:=\frac{X}{\tau(X)}$, then $\tau(A)=1$. By hypothesis, there are mutually orthogonal nonzero projections $E_1,E_2,...,E_N$ such that $\sum_{k=1}^N E_k=I$ and
$\sum_{k=1}^N E_kAE_k=I$. Set $U:=\sum_{k=1}^N e^{\frac{2\pi ki}{N}}E_k$. Then
\[
U^*=\sum_{k=1}^N e^{-\frac{2\pi ki}{N}}E_k=U^{-1}
\]
and
\[
\begin{aligned}
\frac{1}{N}\sum_{j=0}^{N-1}{U^*}^jAU^j
&= \frac{1}{N} \sum_{j=0}^{N-1}\sum_{k,l=1}^N e^{\frac{2\pi i (l-k)j}{N}}E_k A E_l \\
&=\frac{1}{N} \sum_{k,l=1}^N \left(\sum_{j=0}^{N-1}e^{\frac{2\pi i (l-k)j}{N}}\right)E_k A E_l\\
&= \sum_{k=1}^N E_k A E_k=I
\end{aligned}
\]
since $\sum_{j=0}^{N-1}e^{\frac{2\pi i (l-k)j}{N}}=0$ for $k \neq l$. These show that $U$ is the required operator.

(2)$\Rightarrow$(1). Assume that $U$ is the unitary operator with $\frac{1}{N}\sum_{j=0}^{N-1}{U^*}^jAU^j=I$. Then
\[
A-{U^*}^N A U^N=\sum_{j=0}^{N-1}{U^*}^jAU^j-\sum_{j=1}^N {U^*}^jAU^j=NI-U^*\left(\sum_{j=0}^{N-1}{U^*}^jAU^j\right)U=0.
\]
This implies that $AU^N=U^NA$. Let $f(e^{i\theta})=e^{\frac{i\theta}{N}}$ for $\theta \in [0,2\pi)$ and let $W=f(U^N)^{-1}U$. Then $W$ is a unitary operator. Since $U$ commutes with $U^N$, $W^N=(f(U^N)^{-1}U)^N=f(U^N)^{-N}U^N=U^{-N} U^N=I$. Then the spectrum of $W$ consists of eigenvalues from the set $\{e^{\frac{2k
\pi i}{N}}:1\leq k \leq N\}$. Choose $E_1,E_2,...,E_N \in \M$ to be spectral projections of $W$ such that $\sum_{k=1}^NE_k=I$ and
$W= \sum_{k=1}^N e^{\frac{2k \pi i}{N}} E_k$. As $A$ commutes with $U^N$, we have
\[
\frac{1}{N}\sum_{j=0}^{N-1}{W^*}^j A W^j=\frac{1}{N}\sum_{j=0}^{N-1}{U^*}^jf(U^N)^jAf(U^N)^{-j}U^j
=\frac{1}{N}\sum_{j=0}^{N-1}{U^*}^jAU^j=I.
\]
Since $WE_k=e^{\frac{2k \pi i}{N}}E_k$ for each $1\leq k \leq N$, we have
\[
\begin{aligned}
I &=\frac{1}{N}\sum_{k=1}^N E_k \left( \sum_{j=0}^{N-1}{W^*}^j A W^j\right)E_k
   =\frac{1}{N}\sum_{k=1}^N \left(\sum_{j=0}^{N-1}(e^{-\frac{2jk \pi i}{N}} E_k)A(e^{\frac{2jk \pi i}{N}} E_k)\right)\\
  &=\frac{1}{N}\sum_{k=1}^N\left(\sum_{j=0}^{N-1}E_k A E_k\right)=\sum_{k=1}^N E_k A E_k.
\end{aligned}
\]
This proves (1).
\end{proof}

\begin{lem}\label{lem2.2:TFAE}
Let $A\in \M^+$, $\tau(A)=1$ and $N\in \NNN^*$. Then the following conditions are equivalent.
\begin{enumerate}
\item $A$ is the sum of $N$ nonzero projections;
\item There is a decomposition of the identity into $N$ mutually orthogonal nonzero projections $E_j$, $I=\sum_{j=1}^NE_j$, for which $\sum_{j=1}^N E_jAE_j=I$;
\item There is a finite dimensional abelian von Neumann algebra $\B \subseteq \M$ such that
      \[E_{\B^\prime \cap \M}(A)=I, \]
      where $E_{\B^\prime \cap \M}$ is the conditional expectation from $\M$ onto $\B^\prime \cap \M$ preserving $\tau$;
\item There is a unitary operator $U \in \M$ with
\[\frac{1}{N}\sum_{j=0}^{N-1}{U^*}^jAU^j=I.\]
\end{enumerate}
\end{lem}
\begin{proof}
By Lemma \ref{lem2.1:TFAE}, we have (2)$\Leftrightarrow$(4).

(2)$\Rightarrow$(1). $\sum_{j=1}^NE_jAE_j=I$ is equivalent to $E_jAE_j=E_j$ for each $1\leq j\leq N$. Let $P_j=A^{1/2}E_jA^{1/2}$ for $1\leq j\leq N$. Then $P_j$ is a projection for $1\leq
j\leq N$ and $A=\sum_{j=1}^NP_j$.

(1)$\Rightarrow$(2). Let $A=\sum_{j=1}^N P_j$, where $\{P_j\}$ are nonzero projections. Then $1=\tau(A)=\sum_{j=1}^N\tau(P_j)$. We decompose the identity $I=\sum_{j=1}^N F_j$ into $N$
mutually orthogonal projections $F_j\sim P_j$.  Now choose partial isometries $W_j\in \M$ with $P_j=W_j^*W_j$ and $F_j=W_jW_j^*$. Denote $B:=\sum_{j=1}^N W_j$, then
\[
B^*B=\left(\sum_{j=1}^N W_j\right)^*\left(\sum_{j=1}^N W_j\right)=\sum_{i,j=1}^NW_i^*W_j=\sum_{j=1}^NW_j^*W_j=\sum_{j=1}^N P_j=A.
\]
Let $B=VA^{1/2}$. Then $BB^*=VAV^*$. Since $\M$ is a ${\rm II}_1$ factor, $V$ can be chosen to be a unitary operator. Moreover, $F_jB=W_j$ for every $j$, thus
\[
\sum_{j=1}^NF_jVAV^*F_j=\sum_{j=1}^NF_jBB^*F_j=\sum_{j=1}^NW_jW_j^*=\sum_{j=1}^NF_j=I.
\]
Let $E_j=V^*F_jV$. Then we have (2).

(2)$\Rightarrow$(3). Take $\B$ to be the algebra generated by $\{E_j\}_{j=1}^N$. Since $E_{\B^\prime \cap \M}(A)=\sum_{j=1}^NE_jAE_j$. Then, by hypothesis, (3) follows.

(3)$\Rightarrow$(2). Let $E_1,E_2,...,E_N \in \B$ be the minimal projections of $\B$ such that $\sum_{j=1}^N E_j=I$. Then $E_{\B^\prime \cap \M}(A)=\sum_{j=1}^N E_jAE_j=I$.

\end{proof}

\begin{rem}\label{(1)(2)}
Note that in the above Lemma, (1)$\Leftrightarrow$(2) is true for $N=+\infty$.
\end{rem}

\begin{lem}\label{lem2.3: one projection average}
Let $E \in \M$ be a projection. Then there is a unitary operator $U\in \M$ and $N \in \NNN^*$ such that
\[\frac{1}{N}\sum_{j=0}^{N-1}{U^*}^jEU^j=\tau(E)I.\]
\end{lem}

\begin{proof}
The case $E=0,I$ being trivial, we assume that $0<E<I$. Then $0<\tau(E)<1$ and $\frac{1}{\tau(E)}>1$. Set $A:=\frac{1}{\tau(E)}E$. By Lemma \ref{lem1: lemma 6.1}, $A$ is a sum of $N\in \NNN^*$
nonzero
projections in $\M$. Since $\tau(A)=1$, then by Lemma \ref{lem2.2:TFAE}, there is a unitary operator $U\in\M$ such that
\[\frac{1}{N}\sum_{j=0}^{N-1}{U^*}^jEU^j=\tau(E)I.
\]
\end{proof}

\begin{lem}\label{lem2.4: two orth proj tau A= tau RA finite sums}
Let $E$ and $F$ be mutually orthogonal projections in $(\M,\tau)$, $\mu \geq 0$, $0\leq \lambda \leq 1$, and $A:=(1+\mu)E + (1-\lambda)F$. If $\mu \tau(E)\geq \lambda \tau(F)$, then $A$ is a finite sum of projections in $\M$.
\end{lem}
\begin{proof}
By considering $(E+F)\M(E+F)$, we may assume that $E+F=I$. By Lemma \ref{lem1: lemma 6.1}, if $\mu \tau(E) > \lambda \tau(F)$, then $A$ is a finite sum of projections in $\M$. So we may
assume that $\mu \tau(E)= \lambda \tau(F)$. Notice that $A=(1+\mu)E + (1-\lambda)F=(1+\mu)E + (1-\lambda)(I-E)$ and $\tau(A)=1+\mu\tau(E)-\lambda\tau(I-E)=1$. By Lemma
\ref{lem2.3: one projection average}, there is a unitary operator $U$ and $N \in \NNN^*$ such that $\frac{1}{N}\sum_{j=0}^{N-1}{U^*}^jEU^j=\tau(E)I$. Thus
\[
\frac{1}{N}\sum_{j=0}^{N-1}{U^*}^jAU^j=\frac{1}{N}\sum_{j=0}^{N-1}{U^*}^j(I-\lambda I+(\mu+\lambda)E)U^j
=I-\lambda I+(\mu+\lambda)\tau(E)I=I.
\]
So by Lemma \ref{lem2.2:TFAE}, $A$ is a sum of $N$ nonzero projections in $\M$.
\end{proof}

The following lemma is inspired by~\cite{Had}.

\begin{lem}\label{lem2.5: twe proj construction}
Let $E,F$ be two nonzero projections in $(\M,\tau)$. Then there are projections $E_1,E_2,F_1,F_2$ in $\M$ satisfying
\begin{enumerate}
\item $E_1+E_2=E$, $F_1+F_2=F$;
\item $\tau(E_1)=\tau(E_2)=\frac{1}{2}\tau(E)$, $\tau(F_1)=\tau(F_2)=\frac{1}{2}\tau(F)$;
\item $E_1\bot F_1$, $E_2 \bot F_2$.
\end{enumerate}
\end{lem}

\begin{proof}
By the relative position of two projections in a type $\rm{II}_1$ factor (see page 306-308 of~\cite{Tak}), we can write
\[E=
\begin{pmatrix}
P & \\
  & Q \\
  &  &0 \\
  &  & & \begin{pmatrix}
  I &0 \\
  0 &0
  \end{pmatrix}
\end{pmatrix}, \quad
F=
\begin{pmatrix}
P & \\
  & 0 \\
  &  & R \\
  &  & & \begin{pmatrix}
  h^2 & hk \\
  hk & k^2
  \end{pmatrix}
\end{pmatrix},
\]
where $P,Q,R$ are projections, $h \geq 0$, $k\geq 0$ and $h^2+k^2=I$. $h,k$ are contained in a diffuse abelian von Neumann subalgebra $\A$. Hence there is a projection $e\in \A$ with
$\tau_{I}(e)=\frac{1}{2}$. Let $P_1,P_2,Q_1,Q_2,R_1,R_2$ be projections satisfying
\begin{enumerate}
\item $P_1+P_2=P$, $Q_1+Q_2=Q$, $R_1+R_2=R$;
\item $\tau(P_1)=\tau(P_2)=\frac{1}{2}\tau(P)$, $\tau(Q_1)=\tau(Q_2)=\frac{1}{2}\tau(Q)$, $\tau(R_1)=\tau(R_2)=\frac{1}{2}\tau(R)$.
\end{enumerate}
Set
\[E_1=
\begin{pmatrix}
P_1 & \\
  & Q_1 \\
  &  &0 \\
  &  & & \begin{pmatrix}
  e &0 \\
  0 &0
  \end{pmatrix}
\end{pmatrix}, \quad
E_2=
\begin{pmatrix}
P_2 & \\
  & Q_2 \\
  &  &0 \\
  &  & & \begin{pmatrix}
  I-e &0 \\
  0 &0
  \end{pmatrix}
\end{pmatrix},
\]

\[F_1=
\begin{pmatrix}
P_2 & \\
  & 0 \\
  &  &R_1 \\
  &  & & \begin{pmatrix}
  (I-e)h^2 & (I-e)hk \\
  (I-e)hk & (I-e)k^2
  \end{pmatrix}
\end{pmatrix}, \quad
F_2=
\begin{pmatrix}
P_1 & \\
  & 0 \\
  &  &R_2 \\
  &  & & \begin{pmatrix}
  eh^2 & ehk \\
  ehk & ek^2
  \end{pmatrix}
\end{pmatrix}.
\]
Then $E_1,E_2,F_1,F_2$ satisfy the lemma.
\end{proof}

\begin{lem}\label{lem2.6: two proj non orth finite sum}
Let $E,F$ be two nonzero projections in $(\M,\tau)$ and let $X=(1+\mu)E+(1-\lambda)F$, where $\mu \geq 0$ and $0\leq \lambda \leq 1$. If $\mu \tau(E)\geq \lambda
\tau(F)$, then $X$ is a finite sum of projections in $\M$.

\end{lem}
\begin{proof}
To avoid triviality, assume that $\mu > 0$ and $0 < \lambda < 1$. By Lemma \ref{lem2.5: twe proj construction}, there are projections $E_1,E_2,F_1,F_2 \in \M$ satisfying
\begin{enumerate}
\item $E_1+E_2=E$, $F_1+F_2=F$;
\item $\tau(E_1)=\tau(E_2)=\frac{1}{2}\tau(E)$, $\tau(F_1)=\tau(F_2)=\frac{1}{2}\tau(F)$;
\item $E_1\bot F_1$, $E_2 \bot F_2$.
\end{enumerate}
Set $X_1:=(1+\mu)E_1+(1-\lambda)F_1$ and $X_2:=(1+\mu)E_2+(1-\lambda)F_2$. By the hypothesis,
\[
\mu\tau(E_1)=\frac{1}{2}\mu\tau(E)\geq\frac{1}{2}\lambda\tau(F)=\lambda\tau(F_1);
\]
\[
\mu\tau(E_2)=\frac{1}{2}\mu\tau(E)\geq\frac{1}{2}\lambda\tau(F)=\lambda\tau(F_2).
\]
So by Lemma \ref{lem2.4: two orth proj tau A= tau RA finite sums}, $X_1$ and $X_2$ are finite sums of projections in $\M$, respectively, and $X=X_1+X_2$ is a finite sum of projections in
$\M$.
\end{proof}

\begin{lem}\label{lem2.7: finite non orth proj finite sum}
Let $E_i,F_j$, $1 \leq i \leq N, 1\leq j \leq M$ be nonzero projections in $(\M,\tau)$ and let $X=\sum_{i=1}^N(1+\mu_i)E_i+\sum_{j=1}^M(1-\lambda_j)F_j$, where
$\mu_i \geq 0$ and $0\leq \lambda_j \leq 1$. If $\sum_{i=1}^N \mu_i \tau(E_i)\geq \sum_{j=1}^M\lambda_j \tau(F_j)$, then $X$ is a finite sum of projections in $\M$.

\end{lem}
\begin{proof}
(1) If $\mu_1\tau(E_1)=\lambda_1\tau(F_1)$, then by Lemma \ref{lem2.6: two proj non orth finite sum}, $(1+\mu_1)E_1+(1-\lambda_1)F_1$ is a finite sum of projections and $\sum_{i=2}^N \mu_i
\tau(E_i)\geq \sum_{j=2}^M\lambda_j \tau(F_j)$.

(2) If $\mu_1\tau(E_1)>\lambda_1\tau(F_1)$, then there exists a nonzero projection $E_1^\prime\leq E_1$ such that $\mu_1\tau(E_1^\prime)=\lambda_1\tau(F_1)$. Hence
$(1+\mu_1)E_1^\prime+(1-\lambda_1)F_1$ is a finite sum of projections and $\sum_{i=2}^N \mu_i \tau(E_i)+\mu_1\tau(E_1-E_1^\prime)\geq \sum_{j=2}^M\lambda_j \tau(F_j)$.

(3) If $\mu_1\tau(E_1)<\lambda_1\tau(F_1)$, then there exists a nonzero projection $F_1^\prime\leq F_1$ such that $\mu_1\tau(E_1)=\lambda_1\tau(F_1^\prime)$. Hence
$(1+\mu_1)E_1+(1-\lambda_1)F_1^\prime$ is a finite sum of projections and $\sum_{i=2}^N \mu_i \tau(E_i)\geq \sum_{j=2}^M\lambda_j \tau(F_j)+\lambda_1\tau(F_1-F_1^\prime)$.

By induction, we have that $X$ is a finite sum of projections in $\M$.
\end{proof}

\begin{lem}\label{lem2.8:KX finite sum}
Let $E_i$, $1 \leq i \leq K,$ be nonzero projections in $(\M,\tau)$ and let $X=\sum_{i=1}^K\alpha_iE_i$, where $\alpha_i \geq 0$, such that $\tau(X)=1$. Then $KX$
is a finite sum of projections in $\M$.
\end{lem}
\begin{proof}
Since
\[
K=\tau(KX)=\sum_{i=1}^K(K\alpha_i)\tau(E_i)=\sum_{i=1}^K\tau(E_i)+\sum_{i=1}^K(K\alpha_i-1)\tau(E_i)\leq K+\sum_{i=1}^K(K\alpha_i-1)\tau(E_i),
\]
we have $\sum_{i=1}^K(K\alpha_i-1)\tau(E_i)\geq 0$. Without loss of generality, assume that
\[
KX=\sum_{i=1}^N(1+\mu_i)E_i+\sum_{j=1}^M(1-\lambda_j)F_j,
\]
where $\mu_i \geq 0, 0\leq \lambda_j \leq 1$ and $N+M=K$. Then $\sum_{i=1}^N\mu_i\tau(E_i)\geq \sum_{j=1}^M\lambda_j\tau(F_j)$. By Lemma \ref{lem2.7: finite non orth proj finite sum}, $KX$
is
a finite sum of projections in $\M$.

\end{proof}

\begin{rem}\label{rem2.10: x geq 0 constru}

Let $x\in \M^+$. Set $x_1:=x \chi_{x}[1,+\infty)$, $x_2:=x \chi_x[0,1)$, $p:=\chi_{x}[1,+\infty)$ and $q:=\chi_x[0,1)$, then $x=x_1+x_2$. To avoid triviality, we may assume that $x_1,x_2 \neq
0$. Notice that
\[
x_1\geq p, \ \  \  \  x_2 \leq q.
\]
Thus $\tau_p(x_1):=1+\mu\geq1$ and $\tau_q(x_2):=1-\lambda \leq 1$.
So by Lemma \ref{lem1: lemma 3.2}, $\frac{x_1}{1+\mu}=\sum_{i=1}^\infty p_i$ for some projections $\{p_i\}_{i=1}^\infty$ in $p\M p$ and $\frac{x_2}{1-\lambda}=\sum_{j=1}^\infty q_j$ for
some
projections $\{q_j\}_{j=1}^\infty$ in $q\M q$. Therefore,
\[
x=x_1+x_2=(1+\mu)\sum_{i=1}^\infty p_i+ (1-\lambda)\sum_{j=1}^\infty q_j.
\]
The following three lemmas show that $\{p_i\}_{i=1}^\infty$ and $\{q_j\}_{j=1}^\infty$ have very nice properties.
\end{rem}

\begin{lem}\label{q_j property}Let $q_j$ be as in Remark \ref{rem2.10: x geq 0 constru}. Then
$\forall N \in \NNN$, $(1-\lambda)\sum_{j=N+1}^\infty q_j \leq \bigvee_{j=N+1}^\infty q_j$.
\end{lem}

\begin{proof}
For each $N\in \NNN$, $(1-\lambda)\sum_{j=N+1}^\infty q_j \leq (1-\lambda)\sum_{j=1}^\infty q_j=x_2$. Then
\[
\begin{aligned}
(1-\lambda)\sum_{j=N+1}^\infty q_j
 & =\left( \bigvee_{j=N+1}^\infty q_j\right) \left( (1-\lambda)\sum_{j=N+1}^\infty q_j \right) \left( \bigvee_{j=N+1}^\infty q_j\right)\\
& \leq \left( \bigvee_{j=N+1}^\infty q_j\right) x_2 \left( \bigvee_{j=N+1}^\infty q_j\right)\\
&\leq \left( \bigvee_{j=N+1}^\infty q_j\right) q \left( \bigvee_{j=N+1}^\infty q_j\right)\\
&=  \bigvee_{j=N+1}^\infty q_j.
\end{aligned}
\]
\end{proof}

\begin{lem}\label{p_i property}Let $p_i$ be as in Remark \ref{rem2.10: x geq 0 constru}. Then
$\forall N \in \NNN$, $(1+\mu)\sum_{i=N+1}^\infty p_i \geq \bigvee_{i=N+1}^\infty p_i$.
\end{lem}
\begin{proof}
Set $z_1:=\frac{x_1}{1+\mu}$. Then $\tau(z_1)=1$. By Lemma~\ref{lem1: lemma 3.2}, Remark~\ref{(1)(2)}, and $(2)\Rightarrow (1)$ of the proof of Lemma~\ref{lem2.2:TFAE}, $p_i=z_1^{\frac{1}{2}}f_i z_1^{\frac{1}{2}}$ for some mutually orthogonal projections
$f_i$.
Set projection $f:=\sum_{i=N+1}^\infty f_i\leq p$, then we have that
\[
(1+\mu)\sum_{i=N+1}^\infty p_i=(1+\mu)\sum_{i=N+1}^\infty z_1^{\frac{1}{2}}f_i z_1^{\frac{1}{2}}=
\sum_{i=N+1}^\infty x_1^{\frac{1}{2}}f_i x_1^{\frac{1}{2}}=x_1^{\frac{1}{2}} f x_1^{\frac{1}{2}}.
\]
Since $fx_1f \geq fpf\geq f$, $\sigma(f x_1f) \subseteq \{0\} \cup [1,+\infty)$. As $x_1^{\frac{1}{2}} f x_1^{\frac{1}{2}}$ and $fx_1f$ are unitary equivalent, then $\sigma(x_1^{\frac{1}{2}} f
x_1^{\frac{1}{2}}) \subseteq \{0\} \cup [1,+\infty)$. Since $x_1^{\frac{1}{2}} f x_1^{\frac{1}{2}}$ restricted to $R_{(x_1^{\frac{1}{2}} f x_1^{\frac{1}{2}})}H$ is injective, $\sigma\left(x_1^{\frac{1}{2}} f
x_1^{\frac{1}{2}}|_{R_{(x_1^{\frac{1}{2}} f x_1^{\frac{1}{2}})}H}\right)\subseteq[1,+\infty)$.
Since $R(x_1^{\frac{1}{2}} f x_1^{\frac{1}{2}})=R(\sum_{i=N+1}^\infty p_i)=\bigvee_{i=N+1}^\infty p_i$,
 $(1+\mu)\sum_{i=N+1}^\infty p_i=x_1^{\frac{1}{2}} f x_1^{\frac{1}{2}} \geq \bigvee_{i=N+1}^\infty p_i$.

\end{proof}

\begin{lem} \label{p_i,q_j property}Let $p_i$ and $q_j$ be as in Remark \ref{rem2.10: x geq 0 constru}. Then
we have the following properties.
\begin{enumerate}
\item $p=\chi_{x}[1,+\infty)=\bigvee_{i=1}^\infty p_i$;
\item $q=\chi_{x}[0,1)=\bigvee_{j=1}^\infty q_j$;
\item $\forall N \in \NNN$, $\sum_{i=N+1}^\infty \tau(p_i)= \tau\left(\bigvee_{i=N+1}^\infty p_i \right);$
\item $\forall N \in \NNN$, $\sum_{j=N+1}^\infty \tau(q_j)= \tau\left(\bigvee_{j=N+1}^\infty q_j \right).$
\end{enumerate}
\end{lem}

\begin{proof}
We only give the proofs of (1) and (3). The proofs of (2) and (4) are similar.

(1). Since $(\bigvee_{i=1}^\infty p_i)x_1=x_1$, then $\chi_{x}[1,+\infty) \leq \bigvee_{i=1}^\infty p_i$. As $p_i\leq \chi_{x}[1,+\infty)$ for all $i$, thus $\bigvee_{i=1}^\infty p_i \leq
\chi_{x}[1,+\infty)$. So $\chi_{x}[1,+\infty)=\bigvee_{i=1}^\infty p_i$.

(3). By (1), we have that
\[
\sum_{i=1}^\infty \tau_p(p_i)=\tau_p\left(\frac{x_1}{1+\mu}\right)=1=\tau_p(\chi_{x}[1,+\infty))=\tau_p(\bigvee_{i=1}^\infty p_i).
\]
If $\tau\left(\bigvee_{i=N+1}^\infty p_i \right) < \sum_{i=N+1}^\infty \tau(p_i) $ for some $N$, then
$\tau_p\left(\bigvee_{i=N+1}^\infty p_i \right) < \sum_{i=N+1}^\infty \tau_p(p_i) $. Therefore,
\[
1=\tau_p\left(\bigvee_{i=1}^\infty p_i \right)\leq \tau_p(p_1)+\cdots+\tau_p(p_{N})+\tau_p\left(\bigvee_{i=N+1}^\infty p_i \right) <\sum_{i=1}^\infty \tau_p(p_i)=1,
\]
which is a contradiction. Then we have (3).

\end{proof}

\begin{rem}\label{rem2.10': x geq 0 constru}

Let $x\in \M^+$ with $\tau(x)\geq 1$. Set $x_1:=x \chi_{x}[1,+\infty)$, $x_2:=x \chi_x[0,1)$, $p:=\chi_{x}[1,+\infty)$ and $q:=\chi_x[0,1)$, then $x=x_1+x_2$.

Case 1: $\tau(x)=1$. If $\tau_p(x_1)=1$, then $x_1=p$ since $x_1\geq p$. In this case, $\tau_q(x_2)=1$ and $x_2=q$. Thus $x=p+q$ is a finite sum of projections. So in the following we will
assume that $\mu>0$ and therefore $0<\lambda\leq 1$.  Then we have that $\tau_p(\frac{x_1}{1+\mu})=1$ and
$\tau_q(\frac{x_2}{1-\lambda})=1$.
Since $\tau(x)=1$, $\tau\left(\sum_{i=1}^\infty p_i\right)=\tau(p)$ and $\tau\left(\sum_{j=1}^\infty q_j\right)=\tau(q)$, we have
\[
1=\tau(x)=\tau(x_1+x_2)=(1+\mu)\sum_{i=1}^\infty \tau(p_i)+(1-\lambda)\sum_{j=1}^\infty \tau(q_j)=1+\mu\sum_{i=1}^\infty \tau(p_i)-\lambda\sum_{j=1}^\infty \tau(q_j).
\]
Thus
\[
\mu\sum_{i=1}^\infty \tau(p_i)=\lambda\sum_{j=1}^\infty \tau(q_j).
\]

Case 2: $\tau(x)>1$. Since $\tau\left(\sum_{i=1}^\infty p_i\right)=\tau(p)$ and $\tau\left(\sum_{j=1}^\infty q_j\right)=\tau(q)$, similarly, we have
\[
\mu\sum_{i=1}^\infty \tau(p_i)>\lambda\sum_{j=1}^\infty \tau(q_j).
\]
\end{rem}

The idea of the proof of the following lemma comes from the proof of Lemma \ref{lem1: lemma 2.8} in ~\cite{HKNZ}.

\begin{lem}\label{lem2.13:1+ mu X posi line comb K depe mu}
Let $X=(1+\mu)E+(1-\lambda)\sum_{i=1}^\infty F_i$, where $\mu>0$, $0 < \lambda \leq 1$, and $E$, $\{F_i\}_{i=1}^\infty$ are nonzero projections in $(\M,\tau)$. If
$E$, $\{F_i\}_{i=1}^\infty$ satisfy the following conditions
\begin{enumerate}
\item $E \bot (\bigvee_{i=1}^\infty F_i);$
\item $\sum_{i=1}^\infty \tau(F_i) \leq \tau(E);$
\item $(1-\lambda)\sum_{i=1}^\infty F_i\leq \bigvee_{i=1}^\infty F_i,$
\end{enumerate}
then $X=\sum_{i=1}^K \alpha_i E_i$, where $\alpha_i \geq 0$ and $E_i$ are some projections in $\M$. Furthermore, $K\leq 15 + \lceil\frac{28}{\mu}\rceil$.
\end{lem}

\begin{proof}
Set $d:=(1+\mu)E$ and $b:=(1-\lambda)\sum_{i=1}^\infty F_i$, hence $X=b+d$. Notice that by hypothesis (3),
\[
\|b\|\leq \|\bigvee_{i=1}^\infty F_i\|\leq 1.
\]
Since $\sum_{i=1}^\infty \tau(F_i) \leq \tau(E)$, we can choose a partial isometry $V\in \M$ for which $V^*V=\bigvee_{i=1}^\infty F_i$ and $E^\prime:=VV^*\leq E.$ Define the projections
\[
q_-:=\begin{pmatrix}
\frac{b}{\|b\|}    &    -\left( \frac{b}{\|b\|} - \frac{b^2}{\|b\|^2}\right)^\frac{1}{2}V^* \\
 -V\left( \frac{b}{\|b\|} - \frac{b^2}{\|b\|^2}\right)^\frac{1}{2} &  E^\prime - \frac{1}{\|b\|} VbV^*
\end{pmatrix}
\]
and
\[
q_+:=\begin{pmatrix}
\frac{b}{\|b\|}    &    \left( \frac{b}{\|b\|} - \frac{b^2}{\|b\|^2}\right)^\frac{1}{2}V^* \\
 V\left( \frac{b}{\|b\|} - \frac{b^2}{\|b\|^2}\right)^\frac{1}{2} &  E^\prime - \frac{1}{\|b\|} VbV^*
\end{pmatrix}.
\]
Then $q_-+q_+=\frac{2}{\|b\|}b+2E^\prime -\frac{2}{\|b\|}VbV^*$, and hence
\[
X=b+d=\frac{\|b\|}{2}q_- +\frac{\|b\|}{2}q_+ +d-\|b\|E^\prime+VbV^*.
\]
Set $A:=d-\|b\|E^\prime+VbV^*$, then $\|A\|\leq \|d\|+\|VbV^*\|\leq 1+\mu+1=2+\mu$.
Since
\[
A=d-\|b\|E^\prime+VbV^* \geq d-\|b\|E^\prime=(1+\mu)E-\|b\|E^\prime \geq (1+\mu-\|b\|)E \geq \mu E,
\]
it follows by Proposition \ref{prop1: proposition 2.3} and Lemma \ref{lem1: lemma 2.4} that the locally invertible element $A$ can be expressed as a positive linear combination of
projections in $\M$ no more
than
\[
N_0+\lceil V_0\left( \|A\|\|A^{-1}\| -1 \right)\rceil +1 \leq 12+\lceil 14\left(\frac{2+\mu}{\mu}-1\right)\rceil +1=13+\lceil \frac{28}{\mu}\rceil.
\]
So $X$ is a positive combination of $K\leq 15+\lceil \frac{28}{\mu}\rceil$ projections in $\M$.
\end{proof}

\begin{lem}\label{lem2.14:1+mu E finite sum}

Let $X=(1+\mu)E+(1-\lambda) \sum_{i=1}^\infty F_i$, where $\mu>0$, $0< \lambda \leq 1$, and $E$, $\{F_i\}_{i=1}^\infty$ are nonzero projections in $(\M,\tau)$. If
$E$, $\{F_i\}_{i=1}^\infty$ satisfy the following conditions
\begin{enumerate}
\item $E \bot (\bigvee_{i=1}^\infty F_i);$
\item $\lambda\sum_{i=1}^\infty \tau(F_i) < \mu\tau(E);$
\item $\forall N\in \NNN, (1-\lambda)\sum_{i=N+1}^\infty F_i\leq \bigvee_{i=N+1}^\infty F_i,$
\end{enumerate}
then $X$ is a finite sum of projections in $\M$.
\end{lem}
\begin{proof}
Set $\varepsilon:=\mu\tau(E)-\lambda\sum_{i=1}^\infty \tau(F_i)$, then $\varepsilon > 0$ by (2).
Since $\M$ is of type $\rm{II}_1$ factor, we can decompose $E=E_1+E_2$ into the sum of two mutually orthogonal projections such that
\[
0<\tau(E_2)<\frac{\varepsilon}{3(2+\mu)K},
\]
where $K=15+\lceil \frac{28}{\mu}\rceil$. As $\sum_{i=1}^\infty\tau(F_i)<+\infty$, then for $\tau(E_2)$, there exists an $N\in\NNN$ such that
\[
\tau(E_2)\geq \sum_{j=N+1}^\infty\tau(F_j) \ \ \text{and} \ \  \mu\tau(E_2)> \lambda \sum_{j=N+1}^\infty\tau(F_j).
\]
Set $Z:=(1+\mu)E_2+(1-\lambda)\sum_{j=N+1}^\infty F_j$ and $Q:=E_2\vee(\bigvee_{j=N+1}^\infty F_j)$. Then by above inequalities,
\[
s:=\tau_Q(Z)=\frac{\tau(E_2)+\sum_{j=N+1}^\infty \tau(F_j)+\mu\tau(E_2)-\lambda \sum_{j=N+1}^\infty\tau(F_j)}{\tau(Q)} \geq 1.
\]
Thus by Lemma \ref{lem2.13:1+ mu X posi line comb K depe mu}, $Z$ can be expressed as a positive linear combination of no more than $K$ projections and so is $\frac{Z}{s}$. Set
$\frac{Z}{s}=\sum_{j=1}^K\alpha_j q_j$ for some projections $q_j\in \M$ and positive scalars $\alpha_j$. Notice that $\tau_Q(\frac{Z}{s})=1$. Then by Lemma \ref{lem2.8:KX finite sum},
$\frac{KZ}{s}=\sum_{i=1}^n p_i$ for some projections $p_i \in \M$. It implies $KZ=\sum_{i=1}^n sp_i$.
If $s>1$, then by Lemma \ref{lem1: lemma 6.1}, $sp_i$ is a finite sum of projections for every $i$; If $s=1$, then $KZ$ is already a finite sum of projections. So without loss of generality,
we can assume that $KZ=\sum_{i=1}^m f_i$ for some projections $f_i \in \M$.
Thus
\[
\begin{aligned}
X  &=(1+\mu)E_1+(1-\lambda)\sum_{i=1}^N F_i+Z \\
   &=(1+\mu)E_1+(1-\lambda)\sum_{i=1}^N F_i+\frac{1}{K}\sum_{i=1}^m f_i \\
   &=(1+\mu)E_1+(1-\lambda)\sum_{i=1}^N F_i+\left(1-(1-\frac{1}{K})\right)\sum_{i=1}^m f_i.
\end{aligned}
\]
Next we will show that
\[
\mu\tau(E_1)-\lambda\sum_{i=1}^N\tau(F_i)-\left(1-\frac{1}{K}\right)\sum_{i=1}^m \tau(f_i)>0.
\]
For if the above inequality is true, then by Lemma \ref{lem2.7: finite non orth proj finite sum}, $X$ is a finite sum of projections in $\M$.

As
\[
\begin{aligned}
\sum_{i=1}^m \tau(f_i)
& =\tau(KZ)=K\left((1+\mu)\tau(E_2)+(1-\lambda)\sum_{j=N+1}^\infty \tau(F_j)\right)\\
& \leq K\left((1+\mu)\tau(E_2)+\sum_{j=N+1}^\infty \tau(F_j)\right)\\
& \leq K\left((1+\mu)\tau(E_2)+\tau(E_2)\right)<\frac{\varepsilon}{3},
\end{aligned}
\]
thus
\[
\begin{aligned}
&  \mu\tau(E_1)-\lambda\sum_{i=1}^N\tau(F_i)-\left(1-\frac{1}{K}\right)\sum_{i=1}^m \tau(f_i)\\
& >\mu\tau(E)-\mu\tau(E_2)-\lambda\sum_{i=1}^\infty \tau(F_i)-\sum_{i=1}^m \tau(f_i) \\
& >\mu\tau(E)-\lambda\sum_{i=1}^\infty \tau(F_i) -\mu\tau(E_2) -\frac{\varepsilon}{3}\\
& >\varepsilon -\frac{\mu\varepsilon}{3(2+\mu)K}- \frac{\varepsilon}{3} >\frac{\varepsilon}{3}>0.
\end{aligned}
\]
This completes the proof.

\end{proof}

The proof of the following lemma is similar to the proof of  Lemma~\ref{lem2.14:1+mu E finite sum}, however, there exist some differences. So we give a complete proof.

\begin{lem}\label{lem2.15:1- lambda F finite sum}
Let $X=(1-\lambda)F+(1+\mu)\sum_{i=1}^\infty E_i$, where $\mu>0$, $0< \lambda \leq 1$, and $F$, $\{E_i\}_{i=1}^\infty$ are nonzero projections in $(\M,\tau)$. If
$F$, $\{E_i\}_{i=1}^\infty$ satisfy the following conditions
\begin{enumerate}
\item $F \bot (\bigvee_{i=1}^\infty E_i);$
\item $\lambda \tau(F) < \mu\sum_{i=1}^\infty \tau(E_i);$
\item $\forall N\in \NNN, \bigvee_{i=N+1}^\infty E_i \leq (1+\mu)\sum_{i=N+1}^\infty E_i,$
\end{enumerate}
then $X$ is a finite sum of projections in $\M$.
\end{lem}

\begin{proof}
By (3), we have $(1+\mu)\sum_{i=N+1}^\infty E_i$ is locally invertible for every $N$(i.e., invertible in the algebra $(\bigvee_{i=N+1}^\infty E_i)\M(\bigvee_{i=N+1}^\infty E_i)$), thus
$(1+\mu)\sum_{i=N+1}^\infty E_i$ is positive combination of no more than $K_{N+1}:=N_0+ \lceil V_0(\|(1+\mu)\sum_{i=N+1}^\infty E_i\|-1)\rceil+1$ projections, where $N_0=12, V_0=14$, by
Corollary \ref{cor1: Corollary 2.5}. Set $K:=N_0+ \lceil V_0(\|(1+\mu)\sum_{i=1}^\infty E_i\|-1)\rceil+1$. As $(1+\mu)\sum_{i=N+1}^\infty E_i \leq (1+\mu)\sum_{i=1}^\infty E_i$, we have
$K_{N+1} \leq K$ for every $N$. So   $(1+\mu)\sum_{i=N+1}^\infty E_i$  can be expressed as a positive combination of no more than $K$ projections for every $N$.

Set $\varepsilon:=\mu\sum_{i=1}^\infty \tau(E_i)-\lambda \tau(F)$, then $\varepsilon > 0$ by (2). As $\sum_{i=1}^\infty\tau(E_i)<+\infty$, then there exists an $N\in\NNN$ such that
$\sum_{i=N+1}^\infty \tau(E_i)<\frac{\varepsilon}{3(1+\mu)K}$. Set $E:=\bigvee_{i=N+1}^\infty E_i$ and $Z:=(1+\mu)\sum_{i=N+1}^\infty E_i$. Then by (3), we have
$s:=\tau_E(Z)=(1+\mu)\sum_{i=N+1}^\infty \tau_E(E_i)\geq \tau_E(E)=1.$ Therefore $\tau_E(\frac{Z}{s})=1$. Then by Lemma \ref{lem2.8:KX finite sum}, $\frac{KZ}{s}=\sum_{i=1}^n p_i$ for some
projections $p_i \in \M$. It implies $KZ=\sum_{i=1}^n sp_i$. If $s>1$, then by Lemma \ref{lem1: lemma 6.1}, $sp_i$ is a finite sum of projections for every $i$. If $s=1$, then $KZ$ is
already
a finite sum of projections. So without loss of generality, we can assume that $KZ=\sum_{i=1}^m f_i$ for some projections $f_i \in \M$.
Thus
\[
\begin{aligned}
X  &=(1-\lambda)F+(1+\mu)\sum_{i=1}^N E_i+Z \\
   &=(1-\lambda)F+(1+\mu)\sum_{i=1}^N E_i+\frac{1}{K}\sum_{i=1}^m f_i \\
   &=(1-\lambda)F+(1+\mu)\sum_{i=1}^N E_i+\left(1-(1-\frac{1}{K})\right)\sum_{i=1}^m f_i.
\end{aligned}
\]
Next we will show that
\[
\mu\sum_{i=1}^N\tau(E_i)-\lambda\tau(F)-\left(1-\frac{1}{K}\right)\sum_{i=1}^m \tau(f_i)>0.
\]
For if the above inequality is true, then by Lemma \ref{lem2.7: finite non orth proj finite sum}, $X$ is a finite sum of projections in $\M$.

Since
\[
\sum_{i=1}^m \tau(f_i)
 =\tau(KZ)=K\left((1+\mu)\sum_{i=N+1}^\infty \tau(E_i)\right) <\frac{\varepsilon}{3},
\]
we have
\[
\begin{aligned}
&  \mu\sum_{i=1}^N\tau(E_i)-\lambda\tau(F)-\left(1-\frac{1}{K}\right)\sum_{i=1}^m \tau(f_i)\\
& > \mu\sum_{i=1}^\infty \tau(E_i)-\lambda\tau(F)-\mu\sum_{i=N+1}^\infty \tau(E_i)-\frac{\varepsilon}{3}\\
& > \varepsilon -\frac{\mu\varepsilon}{3(1+\mu)K}- \frac{\varepsilon}{3} >\frac{\varepsilon}{3}>0.
\end{aligned}
\]
This completes the proof.
\end{proof}

\begin{lem} \label{lem2.16: > finite sum}
Let $X\in \M^+$. If $\tau(X)>1$, then $X$ is a finite sum of projections in $\M$.
\end{lem}

\begin{proof}
By Remark~\ref{rem2.10: x geq 0 constru}, Remark~\ref{rem2.10': x geq 0 constru}, Lemma~\ref{q_j property}, Lemma~\ref{p_i property} and Lemma~\ref{p_i,q_j property},
 $X=(1+\mu)\sum_{i=1}^\infty E_i+\sum_{j=1}^\infty(1-\lambda)F_j$, where $\mu > 0$, $0< \lambda \leq 1$, and $E_i,F_j$ are nonzero projections in
$(\M,\tau)$ such that
\begin{enumerate}
\item $(\bigvee_{i=1}^\infty E_i) \bot (\bigvee_{j=1}^\infty F_j);$
\item $\mu\sum_{i=1}^\infty \tau(E_i) > \lambda\sum_{j=1}^\infty \tau(F_j);$
\item $\forall N\in \NNN, \bigvee_{i=N+1}^\infty E_i \leq (1+\mu)\sum_{i=N+1}^\infty E_i;$
\item $\forall N\in \NNN, (1-\lambda)\sum_{j=N+1}^\infty F_j\leq \bigvee_{j=N+1}^\infty F_j.$
\end{enumerate}

Set
\[
\varepsilon:=\mu\sum_{i=1}^\infty \tau(E_i)-\lambda \sum_{j=1}^\infty \tau(F_j),
\]
thus $\varepsilon >0$ by (2). Since $\sum_{j=1}^\infty \tau(F_j) < +\infty$, there is a smallest integer $n_1$ for which $0<\mu \tau(E_1)-\lambda \sum_{j=n_1}^\infty \tau(F_j)$. If
$\mu\tau(E_1)-\lambda\sum_{j=n_1}^\infty \tau(F_j) \geq \frac{\varepsilon}{3}$, then we can decompose $F_{n_1-1}=F_{n_1-1}^{\prime}+F_{n_1-1}^{\prime \prime}$ into the sum of two mutually
orthogonal projections with $0<\mu\tau(E_1)-\lambda\sum_{j=n_1}^\infty \tau(F_j)-\lambda \tau(F_{n_1-1}^{\prime}) < \frac{\varepsilon}{3}$. Then without loss of generality, we may assume
that there exists $N,M \in \NNN$ large enough such that
\[
0<\mu\tau(E_1)-\lambda\sum_{j=N+1}^\infty \tau(F_j)<\frac{\varepsilon}{3} \ \ \text{and} \ \ 0<\mu\sum_{i=M+1}^\infty \tau(E_i)-\lambda \tau(F_1)<\frac{\varepsilon}{3}.
\]
Set
\[X_1:=(1+\mu)E_1+(1-\lambda)\sum_{j=N+1}^\infty F_j;\]
\[X_2:=(1-\lambda)F_1+(1+\mu) \sum_{i=M+1}^\infty E_i;\]
\[ X_3:= (1+\mu)\sum_{i=2}^M E_i + (1-\lambda) \sum_{j=2}^N F_j.
\]
Then $X=X_1+X_2+X_3$. By Lemma \ref{lem2.14:1+mu E finite sum} and Lemma \ref{lem2.15:1- lambda F finite sum}, $X_1$ and $X_2$ are finite sums of projections. Now, we consider $X_3$. Since
\[
\begin{aligned}
& \mu \sum_{i=2}^M \tau(E_i) - \lambda\sum_{j=2}^N \tau(F_j)\\
& =\left(\mu \sum_{i=1}^\infty \tau(E_i) - \lambda \sum_{j=1}^\infty \tau(F_j)\right)-\left(\mu \tau(E_1) - \lambda\sum_{j=N+1}^\infty \tau(F_j)\right)-\left(\mu \sum_{i=M+1}^\infty
\tau(E_i) - \lambda \tau(F_1)\right)\\
& >\varepsilon -\frac{\varepsilon}{3}-\frac{\varepsilon}{3}>0.
\end{aligned}
\]
Hence, by Lemma \ref{lem2.7: finite non orth proj finite sum}, $X_3$ is a finite sum of projections. so $X$ is a finite sum of projections.
\end{proof}

\begin{lem}\label{lem2.17: trace= 1+mu finite sum}
Let $X=(1+\mu)E+(1-\lambda) \sum_{i=1}^\infty F_i$, where $\mu>0$, $0< \lambda \leq 1$, and $E$, $\{F_i\}_{i=1}^\infty$ are nonzero projections in
$(\M,\tau)$. If $E$, $\{F_i\}_{i=1}^\infty$ satisfy the following conditions
\begin{enumerate}
\item $E \bot (\bigvee_{i=1}^\infty F_i);$
\item $\lambda\sum_{i=1}^\infty \tau(F_i) = \mu\tau(E);$
\item  $(1-\lambda)\sum_{i=1}^\infty F_i\leq \bigvee_{i=1}^\infty F_i;$
\item  $\sum_{i=1}^\infty \tau(F_i)=\tau(\bigvee_{i=1}^\infty F_i),$
\end{enumerate}
then $X$ is a finite sum of projections in $\M$.
\end{lem}

\begin{proof}
Set $F:=\bigvee_{i=1}^\infty F_i$. By considering $(E+F)\M(E+F)$, we may assume that $E+F=I$. Thus $\tau(X)=1$ by (2) and (4). Set $Z:=\frac{X-I}{\mu}$, then $\tau(Z)=0$ and $\tau(I-Z)=1$.
By (3),
\[
\begin{aligned}
 I-Z &=\left(1+\frac{1}{\mu}\right)F-\left(\frac{1-\lambda}{\mu}\right) \sum_{i=1}^\infty F_i \\
     & \geq \frac{1}{\mu}\left( (1+\mu)(1-\lambda)\sum_{i=1}^\infty F_i -(1-\lambda) \sum_{i=1}^\infty F_i \right)\\
     &= (1-\lambda)\sum_{i=1}^\infty F_i \geq 0
\end{aligned}
\]
and
\[
\begin{aligned}
\tau_F(I-Z) &=\left(1+\frac{1}{\mu}\right)\tau_F(F)-\left(\frac{1-\lambda}{\mu}\right) \sum_{i=1}^\infty \tau_F(F_i) \\
            &  = 1+\frac{1}{\mu}+\frac{\lambda}{\mu} -\frac{1}{\mu}=1+\frac{\lambda}{\mu}>1.
\end{aligned}
\]
Then by Lemma \ref{lem2.16: > finite sum}, $I-Z$ is a finite sum of $N$ projections in $\M$.  So by Lemma \ref{lem2.2:TFAE}, there are
nonzero mutually orthogonal projections $P_1,...,P_N$ such that $\sum_{i=1}^N P_i=I$ and $P_i(I-Z)P_i=P_i $ for all $i=1,2,...,N$, which implies that
$P_iZP_i=0 $ for all $i=1,2,...,N$. Since $X=I+\mu Z$, then we also have $P_iXP_i=P_i(I+\mu Z)P_i=P_i$ for all $i=1,2,...,N$. Notice that
\[
P_i=P_i X^\frac{1}{2} X^\frac{1}{2}P_i=(X^\frac{1}{2}P_i)^*(X^\frac{1}{2}P_i),
\]
for all $i=1,2,...,N$. Thus $X^\frac{1}{2}P_i$ is a partial isometry. So we conclude that $X^\frac{1}{2}P_iX^\frac{1}{2}=Q_i$ is a projection for all $i=1,...,N$. As $\sum_{i=1}^N
P_i=I$, we have $X=Q_1+\cdots+Q_N$.

\end{proof}

\begin{lem}\label{lem2.18: trace = 1- lambda = finite sum}
Let $X=(1-\lambda)F+(1+\mu)\sum_{i=1}^\infty E_i$, where $\mu>0$, $0 < \lambda \leq 1$, and $F$, $\{E_i\}_{i=1}^\infty$ are nonzero projections in $(\M,\tau)$. If $F$, $\{E_i\}_{i=1}^\infty$ satisfy the following conditions
\begin{enumerate}
\item $F \bot (\bigvee_{i=1}^\infty E_i);$
\item $\lambda \tau(F) = \mu\sum_{i=1}^\infty \tau(E_i);$
\item $ \bigvee_{i=1}^\infty E_i \leq (1+\mu)\sum_{i=1}^\infty E_i;$
\item $\sum_{i=1}^\infty \tau(E_i)=\tau(\bigvee_{i=1}^\infty E_i),$
\end{enumerate}
then $X$ is a finite sum of projections in $\M$.
\end{lem}

\begin{proof}
Set $E:=\bigvee_{i=1}^\infty E_i$. By considering $(E+F)\M(E+F)$, we may assume that $E+F=I$. Thus $\tau(X)=1$ by (2) and (4). Set $Z:=\frac{X-I}{\lambda}$, then $\tau(Z)=0$ and
$\tau(I+Z)=1$. Notice that by (3)
\[
\begin{aligned}
 I+Z &=\left(1-\frac{1}{\lambda}\right)E+\left(\frac{1+\mu}{\lambda}\right) \sum_{i=1}^\infty E_i \\
     & \geq \left(1-\frac{1}{\lambda}\right)E  + \frac{1}{\lambda}\bigvee_{i=1}^\infty E_i=E\geq 0
\end{aligned}
\]
and
\[
\begin{aligned}
\tau_E(I+Z) &=\left(1-\frac{1}{\lambda}\right)\tau_E(E)+ \left(\frac{1+\mu}{\lambda}\right) \sum_{i=1}^\infty \tau_E(E_i) \\
            & =\left(1-\frac{1}{\lambda}\right)+\left(\frac{1+\mu}{\lambda}\right)\tau_E\left(\bigvee_{i=1}^\infty E_i\right)  \\
            &  = 1-\frac{1}{\lambda}+ \frac{1+\mu}{\lambda}=1+\frac{\mu}{\lambda}  >1.
\end{aligned}
\]
Then by Lemma \ref{lem2.16: > finite sum}, $I+Z$ is a finite sum of $N$ projections in $\M $. So by Lemma \ref{lem2.2:TFAE}, there are
nonzero mutually orthogonal projections $P_1,...,P_N$ such that $\sum_{i=1}^N P_i=I$ and $P_i(I+Z)P_i=P_i $ for all $i=1,2,...,N$, which implies that
$P_iZP_i=0 $ for all $i=1,2,...,N$. Since $X=I+\lambda Z$, then we also have $P_iXP_i=P_i(I+\lambda Z)P_i=P_i$ for all $i=1,2,...,N$. Notice that
\[
P_i=P_i X^\frac{1}{2} X^\frac{1}{2}P_i=(X^\frac{1}{2}P_i)^*(X^\frac{1}{2}P_i),
\]
for all $i=1,2,...,N$. Then $X^\frac{1}{2}P_i$ is a partial isometry. So we conclude that $X^\frac{1}{2}P_iX^\frac{1}{2}=Q_i$ is a projection for all $i=1,...,N$. Since $\sum_{i=1}^N
P_i=I$, we have $X=Q_1+\cdots+Q_N$.
\end{proof}

\begin{lem}\label{lem2.19: = finite}
Let $X=(1+\mu)\sum_{i=1}^\infty E_i+(1-\lambda) \sum_{j=1}^\infty F_j$, where $\mu>0$, $0 < \lambda \leq 1$, and $E_i,F_j$ are nonzero projections in
$(\M,\tau)$. If $E_i,F_j$ satisfy the following conditions
\begin{enumerate}
\item $(\bigvee_{i=1}^\infty E_i) \bot (\bigvee_{j=1}^\infty F_j);$
\item $\mu\sum_{i=1}^\infty \tau(E_i) = \lambda\sum_{j=1}^\infty \tau(F_j);$
\item $\forall N\in \NNN, \bigvee_{i=N+1}^\infty E_i \leq (1+\mu)\sum_{i=N+1}^\infty E_i;$
\item $\forall N\in \NNN, (1-\lambda)\sum_{j=N+1}^\infty F_j\leq \bigvee_{j=N+1}^\infty F_j,$
\end{enumerate}
then $X$ is a finite sum of projections in $\M$.
\end{lem}

\begin{proof}
Without loss of generality, we may assume that there exists $N,M \in \NNN$ large enough such that
\[
\mu\tau(E_1)=\lambda\sum_{j=N+1}^\infty \tau(F_j) \ \ \text{and} \ \ \mu\sum_{i=M+1}^\infty \tau(E_i)=\lambda \tau(F_1).
\]
Set
\[X_1:=(1+\mu)E_1+(1-\lambda)\sum_{j=N+1}^\infty F_j;\]
\[ X_2:=(1-\lambda)F_1+(1+\mu)\sum_{i=M+1}^\infty E_i; \]
\[X_3:= \sum_{i=2}^M (1+\mu)E_i + (1-\lambda) \sum_{j=2}^N F_j.
\]
Then $X=X_1+X_2+X_3$. By Lemma \ref{lem2.17: trace= 1+mu finite sum} and Lemma \ref{lem2.18: trace = 1- lambda = finite sum}, $X_1$ and $X_2$ are finite sums of projections. Now, we
consider
$X_3$. Since
\[
\begin{aligned}
& \mu \sum_{i=2}^M \tau(E_i) - \lambda\sum_{j=2}^N \tau(F_j)\\
& =\left(\mu \sum_{i=1}^\infty \tau(E_i) - \lambda \sum_{j=1}^\infty \tau(F_j)\right)-\left(\mu \tau(E_1) - \lambda\sum_{j=N+1}^\infty \tau(F_j)\right)-\left(\mu \sum_{i=M+1}^\infty
\tau(E_i) - \lambda \tau(F_1)\right)\\
& =0.
\end{aligned}
\]
Hence, by Lemma \ref{lem2.7: finite non orth proj finite sum}, $X_3$ is a finite sum of projections, so is $X$.
\end{proof}

\begin{lem}\label{lem2.20:tau X=1 X is finite sum}
Let $X\in \M^+$. If $\tau(X)=1$, then $X$ is a finite sum of projections in $\M$.
\end{lem}

\begin{proof}
By Remark \ref{rem2.10: x geq 0 constru}, we can assume that
\[
X=(1+\mu)\sum_{i=1}^\infty p_i+ (1-\lambda)\sum_{j=1}^\infty q_j,
\]
where $\mu > 0, 0 < \lambda \leq 1$ and $p_i,q_j$ satisfy the properties of Lemma \ref{q_j property}, Lemma \ref{p_i property} and Lemma \ref{p_i,q_j property}. Furthermore, we have that
\[
1=\tau(X)=1+\mu\sum_{i=1}^\infty\tau(p_i)-\lambda\sum_{j=1}^\infty \tau(q_j).
\]
It implies that $\mu\sum_{i=1}^\infty\tau(p_i)=\lambda\sum_{j=1}^\infty \tau(q_j)$. Then by Lemma \ref{lem2.19: = finite}, it follows that $X$ is a finite sum of projections in $\M$.

\end{proof}

\begin{thm}\label{thm: main thm}
Let $(\M,\tau)$ be a type ${\rm II}_1$ factor,  $X\in \M$, $X=X^*$. Then we have the following results.
\begin{enumerate}
\item There is a decomposition of the identity into $N\in \NNN$ mutually orthogonal nonzero projections $E_j$, $I=\sum_{j=1}^NE_j$, for which $E_jXE_j=\tau(X) E_j$ for all $j=1,...,N$;
\item There is a finite dimensional abelian von Neumann algebra $\B \subseteq \M$ such that
      \[E_{\B^\prime \cap \M}(X)=\tau(X)I; \]
where $E_{\B^\prime \cap \M}$ is the conditional expectation from $\M$ onto $\B^\prime \cap \M$;
\item There is a unitary operator $U \in \M$ with
\[\frac{1}{N}\sum_{j=0}^{N-1}{U^*}^jXU^j=\tau(X)I;\]
\item There is a unitary operator $W \in \M$ with $W^N=I$ and
\[\frac{1}{N}\sum_{j=0}^{N-1}{W^*}^jXW^j=\tau(X)I.\]
\end{enumerate}
\end{thm}

\begin{proof}
We first consider the case $X\geq 0$. Assume that $\tau(X)=s$. It implies that $\tau(\frac{X}{s})=1$. Then by Lemma \ref{lem2.20:tau X=1 X is finite sum}, $\frac{X}{s}$ is a finite sum of
$N$
projections. Thus by Lemma \ref{lem2.1:TFAE}, there is a unitary operator $U$ with
\[
\frac{1}{N}\sum_{i=0}^{N-1}{U^*}^i \left(\frac{X}{s}\right)U^i=I.
\]
Therefore,
\[
\frac{1}{N}\sum_{i=0}^{N-1}{U^*}^i X U^i=sI.
\]
Now, let $X$ be a self-adjoint operator. Since $-\|X\|I \leq X$, we have $\|X\|I+X\geq0$. By above argument, there exists a unitary operator $U$ with
\[
\frac{1}{N}\sum_{i=0}^{N-1}{U^*}^i (\|X\|I+X)U^i=\tau(\|X\|I+X)I=\|X\|I+sI.
\]
Thus
\[
\frac{1}{N}\sum_{i=0}^{N-1}{U^*}^i X U^i=sI=\tau(X)I.
\]
Therefore, by Lemma \ref{lem2.1:TFAE} and Lemma \ref{lem2.2:TFAE}, we have the theorem.
\end{proof}

The following theorem answers affirmatively Question 6.7 in \cite{HKNZ}.

\begin{thm}
Let $(\M,\tau)$ be a type ${\rm II}_1$ factor and $A\in \M^+$. Then $A$ is a finite sum of projections if and only if $\tau(A)\geq \tau(R_A)$.
\end{thm}

\begin{proof}
$``\Rightarrow"$. Let $A=P_1+\cdots +P_N$ with $P_i$ nonzero projections and $N\in \NNN$. By Kaplansky's parallelogram law [\cite{K-R}, Theorem, 6.1.7], we have
\[
\tau(R_A)=\tau\left(\bigvee_{i=1}^NP_i\right)\leq \sum_{i=1}^N\tau(P_i)=\tau(A).
\]

$``\Leftarrow"$. By considering $R_A\M R_A$, we may assume that $R_A=I$. Then $\tau(A)\geq 1$.
By Lemma \ref{lem2.16: > finite sum} and Lemma \ref{lem2.20:tau X=1 X is finite sum}, it
implies that $A$ is a finite sum of projections in $\M$.

\end{proof}

The following theorem answers affirmatively the Question 1.1 in \cite{DFS}.

\begin{thm}
Let $(\M,\tau)$ be a type ${\rm II}_1$ factor,  $X\in \M$, $X=X^*$ and $\tau(X)=0$. Then there exists a nilpotent element $Z \in \M$ such that $X=Re Z$, i.e., $X$ is the real part of a
nilpotent element $Z$.
\end{thm}
\begin{proof}
By Theorem \ref{thm: main thm}, there are $N$ mutually orthogonal nonzero projections $E_j$ such that $I=\sum_{j=1}^NE_j$ and $E_jXE_j=0$ for all $j=1,...,N$. Since $X$ is self-adjoint,
we
can write
\[
X=\begin{pmatrix}
0&X_{12}& X_{13}& \cdots &X_{1N}\\
X_{12}^*& 0 &X_{23} & \cdots &X_{2N}\\
X_{13}^*& X_{23}^* & 0 & \cdots &X_{3N} \\
\vdots & \vdots & \ddots & \ddots & \vdots \\
X_{1N}^*& X_{2N}^*& X_{3N}^* & \cdots & 0
\end{pmatrix}
\]
with respect to the decomposition $I=\sum_{j=1}^NE_j$. Set
\[
Z:=\begin{pmatrix}
0&X_{12}& X_{13}& \cdots &X_{1N}\\
0 & 0 &X_{23} & \cdots &X_{2N}\\
0& 0 & 0 & \cdots &X_{3N} \\
\vdots & \vdots & \ddots & \ddots & \vdots \\
0 & 0 & 0 & \cdots & 0
\end{pmatrix}.
\]
We have that $Z$ is a nilpotent operator with $X=Z+Z^*$. This shows that $X$ is the real part of $Z$.
\end{proof}

\section{Dixmier's averaging theorem for type ${\rm II}_1$ factors}

\begin{defn}
$A$ is said to be majorized by $B$, in notion $A \prec B$, if there are unitary operators $U_1,...,U_N \in \M$ such that
\[
\frac{1}{N}\sum_{i=1}^N U_i^* BU_i=A.
\]
\end{defn}
\begin{prop}
If $A \prec B$ and $B\prec C$, then $A\prec C$.
\end{prop}
\begin{proof}
If
\[
\frac{1}{N}\sum_{i=1}^N U_i^* BU_i=A \quad \text{and} \quad \frac{1}{M}\sum_{j=1}^M V_j^* CV_j=B,
\]
then
\[
\frac{1}{NM}\sum_{i=1}^N\sum_{j=1}^M U_i^*V_j^* CV_jU_i=A.
\]
\end{proof}

\begin{lem}\label{lem:diagonal non equal majori}
Set $B:=\begin{pmatrix}
E_1BE_1& \cdots &E_1BE_n \\
\vdots&     &  \vdots \\
E_nBE_1  &  \cdots   & E_nBE_n
\end{pmatrix}$
with respect to decomposition $I=E_1+\cdots+E_n$. Then $A\prec B$, where
$A=\begin{pmatrix}
E_1BE_1& \cdots & 0 \\
\vdots&   \ddots  &  \vdots \\
0  &  \cdots   & E_nBE_n
\end{pmatrix}$.
\end{lem}
\begin{proof}
For the $n \times n$ diagonal matrix. Let the first position is 1 and the other position is 1 or -1 with respect to decomposition $I=E_1+\cdots+E_n$. Then there are $2^{n-1}$ unitary
operators denoted by $U_1,...,U_{2^{n-1}}$. It can prove that
\[
\frac{1}{2^{n-1}}\sum_{i=1}^{2^{n-1}}U_i^* B U_i=A.
\]
\end{proof}

\begin{lem}\label{lem:diagonal majori}
If $(E_{a,b})_{a,b\in B}$ is a self-adjoint system of $n\times n$ matrix units for a von Neumann algebra $\M$, it is known that
$\M\cong M_n(\mathds{C}) \otimes \N$ for some von Neumann algebra $\N$. In this algebra,
$A:=\begin{pmatrix}
\frac{A_1+A_2+\cdots+A_n}{n}& \cdots & 0 \\
\vdots&  \ddots   &  \vdots \\
0  &  \cdots   & \frac{A_1+A_2+\cdots+A_n}{n}
\end{pmatrix}
\prec
\begin{pmatrix}
A_1& \cdots & 0 \\
\vdots&  \ddots   &  \vdots \\
0  &  \cdots   & A_n
\end{pmatrix}:=B,$ where $A_i \in \N^+$ for $ i=1,...,n$.
\end{lem}
\begin{proof}
Set $U_1=I$,
$U_i:=
\bordermatrix{%
       & 1         & 2       &\cdots   &i  &i+1 &\cdots           &n\cr
1    & 0         & 0         &\cdots   &1  &0   & \cdots           & 0\cr
2    & 0         & 0         &\cdots   &0  &1   & \cdots           & 0\cr
\vdots & \vdots  &\vdots     &\cdots   &\vdots  &\vdots    &\ddots & \vdots\cr
n+2-i & 1 &0     &\cdots               &0       &0         &\cdots & 0\cr
n+3-i & 0 & 1    &\cdots               &0       &0         &\cdots & 0\cr
\vdots & \vdots  &\vdots     &\cdots   &\vdots  &\vdots    &\ddots & 0\cr
n     & 0    & 0         &\cdots   &0           &0             &\cdots        &0
},$ where $i=2,...,n-1$,
and
$U_n=\begin{pmatrix}
0& \cdots & 1 \\
\vdots&  \ddots   &  \vdots \\
1  &  \cdots   & 0
\end{pmatrix}.$ Then it can prove that
\[
\frac{1}{n}\sum_{i=1}^n U_i^* BU_i=A.
\]
\end{proof}

\begin{prop}\label{prop:unitary equiva if A qrec B then A qrec C}
Let $A,B_1,B_2\in \M$ and $A \prec B_1$. If $B_1$ is unitary equivalent with $B_2$, then $A\prec B_2$.
\end{prop}
\begin{proof}
Since $A \prec B_1$, then by definition, there are unitary operators $U_1,...,U_N \in \M$ such that
\[
\frac{1}{N}\sum_{i=1}^N U_i^* B_1U_i=A.
\]
As $B_1$ is unitary equivalent with $B_2$, then $B_1=V^*B_2 V$ for some unitary operator $V$.
Thus
\[
\frac{1}{N}\sum_{i=1}^N U_i^*V^* B_2VU_i=A.
\]

\end{proof}

The following lemma is a generalization of Corollary 7.3 of~\cite{And}.
\begin{lem}
If $(E_{a,b})_{a,b\in B}$ is a self-adjoint system of $n\times n$ matrix units for a von Neumann algebra $\M$, it is known that
$\M\cong M_n(\mathds{C}) \otimes \N$ for some von Neumann algebra $\N$. In the algebra $\M\cong M_n(\mathds{C}) \otimes \N$,

\[
\begin{pmatrix}
A_1       &  0  & \cdots &0\\
0         & A_2 & \cdots &0\\
\vdots    &\vdots &\ddots &\vdots\\
0         &0    &\cdots  &A_n
\end{pmatrix}
\prec
\begin{pmatrix}
A_1+A_2+\cdots+A_n& 0& \cdots &0\\
0                                    & 0& \cdots &0\\
\vdots                               & \vdots& \ddots &\vdots\\
0                                    &  0&  \cdots  &0
\end{pmatrix},
\] where $A_i \in \N^+$ for $i=1,2,...,n$
\end{lem}

\begin{proof}
First, we consider $2\times2$ matrix. Write
$\begin{pmatrix}
A+B&  0 \\
0    &  0
\end{pmatrix}=
\begin{pmatrix}
A^{\frac{1}{2}}& B^{\frac{1}{2}} \\
0    &  0
\end{pmatrix}
\begin{pmatrix}
A^{\frac{1}{2}}& 0 \\
B^{\frac{1}{2}}    &  0
\end{pmatrix},$ which is unitary equivalent with
$\begin{pmatrix}
A^{\frac{1}{2}}& 0 \\
B^{\frac{1}{2}}    &  0
\end{pmatrix}
\begin{pmatrix}
A^{\frac{1}{2}}& B^{\frac{1}{2}} \\
0    &  0
\end{pmatrix}.
$
By the proof of Lemma \ref{lem:diagonal non equal majori}, $\begin{pmatrix}
A& 0 \\
0    &  B
\end{pmatrix}\prec
\begin{pmatrix}
A& A^{\frac{1}{2}}B^{\frac{1}{2}} \\
B^{\frac{1}{2}}A^{\frac{1}{2}}    &  B
\end{pmatrix}
=\begin{pmatrix}
A^{\frac{1}{2}}& 0 \\
B^{\frac{1}{2}}    &  0
\end{pmatrix}
\begin{pmatrix}
A^{\frac{1}{2}}& B^{\frac{1}{2}} \\
0    &  0
\end{pmatrix}.$ Then by Proposition \ref{prop:unitary equiva if A qrec B then A qrec C}, we have $\begin{pmatrix}
A&  0 \\
0    &  B
\end{pmatrix}
\prec
\begin{pmatrix}
A+B&  0 \\
0    &  0
\end{pmatrix}.$ Next we will prove that
$\begin{pmatrix}
A&  0 &0\\
0    &  B&0\\
0  &0  &C
\end{pmatrix}
\prec
\begin{pmatrix}
A+B+C&  0 &0\\
0    &  0  &0\\
0    &0    &0
\end{pmatrix},$ where $A,B,C \in \N^+$. Similarly to the above argument.
\[\begin{pmatrix}
A+B+C&  0 &0\\
0    &  0  &0\\
0    &0    &0
\end{pmatrix}=
\begin{pmatrix}
(A+B)^{\frac{1}{2}}&  0 &C^{\frac{1}{2}}\\
0    &  0  &0\\
0    &0    &0
\end{pmatrix}
\begin{pmatrix}
(A+B)^{\frac{1}{2}}&  0 &0\\
0    &  0  &0\\
C^{\frac{1}{2}}   &0    &0
\end{pmatrix},
\] which is unitary equivalent with
\[
\begin{pmatrix}
(A+B)^{\frac{1}{2}}&  0 &0\\
0    &  0  &0\\
C^{\frac{1}{2}}   &0    &0
\end{pmatrix}
\begin{pmatrix}
(A+B)^{\frac{1}{2}}&  0 &C^{\frac{1}{2}}\\
0    &  0  &0\\
0    &0    &0
\end{pmatrix}.
\]
Thus
\[\begin{pmatrix}
A+B&  0 &0\\
0    &  0&0\\
0  &0  &C
\end{pmatrix}
\prec
\begin{pmatrix}
A+B&  0 &(A+B)^{\frac{1}{2}}C^{\frac{1}{2}}\\
0    &  0&0\\
C^{\frac{1}{2}}(A+B)^{\frac{1}{2}}  &0  &C
\end{pmatrix}=
\begin{pmatrix}
(A+B)^{\frac{1}{2}}&  0 &0\\
0    &  0  &0\\
C^{\frac{1}{2}}   &0    &0
\end{pmatrix}
\begin{pmatrix}
(A+B)^{\frac{1}{2}}&  0 &C^{\frac{1}{2}}\\
0    &  0  &0\\
0    &0    &0
\end{pmatrix}.
\]
Then we have
\[
\begin{pmatrix}
A+B&  0 &0\\
0    &  0&0\\
0  &0  &C
\end{pmatrix}
\prec
\begin{pmatrix}
A+B+C&  0 &0\\
0    &  0&0\\
0  &0  &0
\end{pmatrix}.
\]
Since by above argument,
\[
\begin{pmatrix}
A&  0 &0\\
0    &  B&0\\
0  &0  &C
\end{pmatrix}
\prec
\begin{pmatrix}
A+B&  0 &0\\
0    &  0&0\\
0  &0  &C
\end{pmatrix}.
\]
So we prove that
\[
\begin{pmatrix}
A&  0 &0\\
0    &  B&0\\
0  &0  &C
\end{pmatrix}
\prec
\begin{pmatrix}
A+B+C&  0 &0\\
0    &  0&0\\
0  &0  &0
\end{pmatrix}.
\]
By the same argument we have that
\[
\begin{pmatrix}
A_1       &  0  & \cdots &0\\
0         & A_2 & \cdots &0\\
\vdots    &\vdots &\ddots &\vdots\\
0         &0    &\cdots  &A_n
\end{pmatrix}
\prec
\begin{pmatrix}
A_1+A_2+\cdots +A_n& 0& \cdots &0\\
0                                    & 0& \cdots &0\\
\vdots                               & \vdots& \ddots &\vdots\\
0                                    &  0&  \cdots  &0
\end{pmatrix}.
\]
\end{proof}

\begin{prop}
Let $A\in \M$ be an injective positive map with dense range. Then the following conditions are equivalent.
\begin{enumerate}
\item
\[
\begin{pmatrix}
B &0\\
0 &0
\end{pmatrix}\prec
\begin{pmatrix}
A & 0 \\
0 & 0
\end{pmatrix};
\]
\item
\[
B \prec A.
\]
\end{enumerate}

\end{prop}

\begin{proof}
$(1)\Rightarrow(2)$. By hypothesis, there are unitary operators $U_1,...,U_N \in \M$ such that
\[
\frac{1}{N}\sum_{i=1}^N U_i^*\begin{pmatrix}
A & 0 \\
0 & 0
\end{pmatrix}
U_i=\begin{pmatrix}
B &0\\
0 &0
\end{pmatrix}.
\]
Set $U_i:=\begin{pmatrix}
U_{11}^i &U_{12}^i\\
U_{21}^i &U_{22}^i
\end{pmatrix}.$ Then we have
$\frac{1}{N}\sum_{i=1}^N (U_{21}^i)^*AU_{21}^i=0.$ Since $(U_{21}^i)^*AU_{21}^i \geq 0$ for every $i=1,...,N$, then
$(U_{21}^i)^*AU_{21}^i=0$ for every $i=1,...,N$. Write $(U_{21}^i)^*AU_{21}^i=(U_{21}^i)^*A^{\frac{1}{2}}A^{\frac{1}{2}}U_{21}^i$, thus $A^{\frac{1}{2}}U_{21}^i=0$. As $A$ is injective,
it implies that $U_{21}^i=0$ for every $i=1,..., N$. Since $U_i$ is a unitary operator, it implies that
$(U_{11}^i)^*U_{11}^i=I$ and $U_{11}^i(U_{11}^i)^*+U_{12}^i(U_{12}^i)^*=I$. As $\M$ is finite factor, then $(U_{11}^i)^*U_{11}^i=I$ implies that $U_{11}^i(U_{11}^i)^*=I$. So we have $U_{12}^i=0$
for every $i=1,...,N$. Therefore we conclude that
\[
\frac{1}{N}(U_{11}^i)^* AU_{11}^i=B.
\]
Then we have (2).
\end{proof}
\begin{cor}
Let $(\M,\tau)$ be a type $\rm {II}_1$ factor and $A\in \M$ be an injective positive map with dense range. Then the following conditions are equivalent.
\begin{enumerate}
\item
\[
\begin{pmatrix}
\tau(A)I &0\\
0 &0
\end{pmatrix}\prec
\begin{pmatrix}
A & 0 \\
0 & 0
\end{pmatrix};
\]
\item
\[
\tau(A)I \prec A.
\]
\end{enumerate}
\end{cor}

The following theorem is a stronger version of Dixmier's averaging theorem for type ${\rm II}_1$ factors.
\begin{thm}
Let $(\M,\tau)$ be a type ${\rm II}_1$ factor and let $X_1,...,X_n\in \M$. Then there exist unitary operators $U_1,...,U_k\in\M$ such that
\[
\frac{1}{k}\sum_{i=1}^kU_i^{*}X_jU_i=\tau(X_j)I,\quad \forall 1\leq j\leq n.
\]
\end{thm}
\begin{proof}
Clearly, we may assume that $X_i=X_i^*$ for $1\leq i\leq n$. For $X_1$, by Theorem~\ref{thm: main thm}, there exist unitary operators $V_1,...,V_r$ such that
\[
\frac{1}{r}\sum_{i=1}^rV_i^{*}X_1V_i=\tau(X_1)I.
\]
Let
\[
A=\frac{1}{r}\sum_{i=1}^rV_i^{*}X_2V_i.
\]
Then  $A=A^*$ and $\tau(A)=\tau(X_2)$. By Theorem~\ref{thm: main thm}, there exist unitary operators $W_1,...,W_s$ such that
\[
\frac{1}{s}\sum_{i=1}^sW_i^{*}AW_i=\tau(X_2)I.
\]
Now it is easy to see that $W_iV_j$ satisfies
\[
\frac{1}{rs}\sum_{i=1}^s\sum_{j=1}^rW_i^{*}V_j^{*}X_k V_jW_i=\tau(X_k)I,\quad k=1,2.
\]
Now, by induction we have the theorem.
\end{proof}

\end{document}